\documentclass[10pt,a4paper]{article}
\usepackage[latin1]{inputenc}
\usepackage{amsmath}
\usepackage{amsfonts}
\usepackage{amssymb}
\usepackage{graphicx}
\usepackage{amsthm} 	
\usepackage{tikz}
\usepackage{hyperref}
\usepackage{todonotes}
\usetikzlibrary{backgrounds}	
\usepackage[toc]{appendix}

\usepackage{graphicx,xcolor}	

\usepackage{xcolor}
\pagecolor{white}

\usepackage{esint}				

\usepackage[a4paper,top=2.1cm,bottom=2.60cm,left=3.1cm,right=3.1cm]{geometry} 

\allowdisplaybreaks

\newtheorem{theorem}{Theorem}[section]
\newtheorem{proposition}[theorem]{Proposition}

\newtheorem{definition}[theorem]{Definition}
\newtheorem{remark}[theorem]{Remark}

\title{Asymptotic analysis a perturbed Robin problem in a planar domain}
\author{Paolo Musolino\thanks{Dipartimento di Scienze Molecolari e Nanosistemi, Universit\`a Ca' Foscari Venezia, via Torino 155, 30172 Venezia Mestre, Italy} ,  Martin Dutko\thanks{Rockfield Software Limited, King's Road, Ethos Building, Swansea, SA1 8PH, Wales UK} ,  Gennady Mishuris\thanks{Department of Mathematics, Aberystwyth University, Ceredigion, Aberystwyth, SY23 3BZ Wales, UK} }

\date{\ }

\begin{document}

\maketitle

\noindent
{\bf Abstract:} We consider a perforated domain $\Omega(\epsilon)$ of $\mathbb{R}^2$ with a small hole of size $\epsilon$ and we study the behavior of the solution of a mixed Neumann-Robin  problem in $\Omega(\epsilon)$ as the size $\epsilon$ of the small hole tends to $0$. In addition to the geometric degeneracy of the problem, the $\epsilon$-dependent Robin condition may degenerate into a Neumann condition for $\epsilon=0$ and the Robin datum may diverge to infinity. Our goal is to analyze the asymptotic behavior of the solutions to the problem as $\epsilon$ tends to $0$ and understand how the boundary condition affects the behavior of the solutions when $\epsilon$ is close to $0$.

\vspace{9pt}

\noindent
{\bf Keywords:}  singularly perturbed boundary value problem,  Laplace equation, nonlinear Robin condition, perforated planar domain, integral equations

\vspace{9pt}

\noindent
{{\bf 2020 Mathematics Subject Classification:} 35J25; 31B10; 35B25; 35C20; 47H30}

\section{Introduction}
\label{introd}
In this paper we continue the analysis of \cite{MuMi22}, where we have studied the asymptotic behavior of the solutions of a boundary value problem for the Laplace equation in a perforated domain in $\mathbb{R}^n$, $n\geq 3$,  with a nonlinear Robin boundary condition degenerating into a Neumann condition on the boundary of the small hole.

The problem considered in  \cite{MuMi22} was degenerating under three aspects: in the limit case the Robin boundary condition may degenerate into a Neumann boundary condition, the Robin datum may tend to infinity, and, finally, the size $\epsilon$ of the small hole where we consider the Robin condition tends to $0$. The analysis of \cite{MuMi22} was confined to the case of dimension $n\geq 3$, since the two-dimensional case requires a different treatment. Indeed the technique of \cite{MuMi22} is based on potential theory, and as it happens often with such method, the case of dimension $n=2$ and the one of dimension $n\geq 3$ need to be treated separately because of the different aspect of the fundamental solution of the Laplacian.

Boundary value problems with degenerating or perturbed boundary conditions have been analyzed by many authors. Here we mention, for example, Wendland, Stephan, and Hsiao \cite{WeStHs79}, Kirsch \cite{Ki85}, Costabel and Dauge \cite{CoDa96},  Ammari and N\'ed\'elec \cite{AmNe99}, Schmidt and Hiptmair \cite{ScHi17}, and \cite{MuMi18, MuMi22}.

 As already mentioned, another feature of the problem considered in the present paper and in \cite{MuMi22} is the fact that the degenerating boundary condition is posed on the boundary of a small hole. Boundary value problems in domain with small holes have been studied by many authors. Asymptotic analysis techniques have been used for example in the works of Ammari  and Kang  \cite{AmKa07},  Il'in \cite{Il92}, Maz'ya, Movchan, and Nieves \cite{MaMoNi13, MaMoNi14, MaMoNi16, MaMoNi17, MaMoNi21}, Maz'ya, Nazarov, and Plamenevskij \cite{MaNaPl00i, MaNaPl00ii}, Nieves \cite{Ni17}, Nieves and Movchan \cite{NiMo22}, Novotny and Soko{\l}{o}wski \cite{NoSo13}. The method of the present paper is instead, as in \cite{MuMi22}, the Functional Analytic Approach proposed by Lanza de Cristoforis in \cite{La02} for the analysis of singular perturbation problem in perforated domains. The purpose of the method is to represent the solution of a pertubed problem in terms of real analytic maps and known functions of the perturbation parameters. In particular, we observe that such method has been successfully used for example in Dalla Riva and Lanza de Cristoforis \cite{DaLa10, DaLa10bis, DaLa10ter, DaLa11} and Lanza de Cristoforis \cite{La07}, for the analysis of nonlinear boundary value problems.

In scientific and engineering practice, Robin boundary condition has an important role in many applications. Perhaps the most common use are the transport PDEs utilised in the systems such as convective-dispersive solute transport (van Genuchten and  Alves \cite{vaAl82}), heat transfer (e.g. temperature dependent boundary conditions in forming of the glass containers ass seen  at  \cite{ElfenGlass}), and convective-diffusive mass transfer of different species.   Here, the ability to define arbitrary size of the internal perturbation with Robin boundary condition is important when assessing processes at different scales --  for example when analysing sand fines migration from or into the well during oil or gas production the size of the perturbation $\delta$ (wellbore diameter) will be finite at the wellbore scale assessment but $\delta\to 0$  for field scale analysis (see e.g. \cite{ElfenWelbore}  for various Oil and Gas applications).
 In \cite{MuMi18, MuMi22} and in the present paper, we have considered a Robin problem as simplified model for the transmission problem for a composite domain with imperfect conditions along the joint boundary. Such nonlinear transmission conditions frequently appear in practical applications for various nonlinear multiphysics problems ({\it e.g.}, \cite{Mi04,MiMiOc08,MiMiOc09}).

We begin by introducing the geometry of our problem. Therefore, we fix a regularity parameter $\alpha\in]0,1[$ and we take two subsets, one representing the the unperturbed domain $\Omega^o$ and another representing the shape of the hole $\omega^i$. The sets   $\Omega^o$ and $\omega^i$   satisfy the following assumption
\[
\begin{split}
&\text{$\omega^i$ and $\Omega^o$ are bounded open connected subsets of $\mathbb{R}^2$ of class $C^{1,\alpha}$}\\
&\text{such that $0 \in  \Omega^o \cap \omega^i$ and that $\mathbb{R}^2\setminus\overline{\omega^i}$ and $\mathbb{R}^2\setminus\overline{\Omega^o}$ are connected}.
\end{split}
\]
We refer to Gilbarg and Trudinger~\cite{GiTr83} for the definition of sets and functions of the Schauder class $C^{k,\alpha}$ ($k \in \mathbb{N}$). We set
\[
\epsilon_0 \equiv \mbox{sup}\{\theta \in \mathopen]0, +\infty\mathclose[: \epsilon \overline{\omega^i} \subseteq \Omega^o, \ \forall \epsilon \in \mathopen]- \theta, \theta[  \}\, .
\]
If $\epsilon \in ]0,\epsilon_0[$, then the set $\epsilon \overline{\omega^i}$ is contained in $\Omega^o$. We think of  $\epsilon \overline{\omega^i}$ as a hole and we remove it from the unperturbed domain. Hence, we introduce the perforated domain $\Omega(\epsilon)$ by setting
\[
\Omega(\epsilon) \equiv \Omega^o \setminus \epsilon \overline{\omega^i} \qquad \forall \epsilon\in\mathopen]0,\epsilon_0[\, .
\]
As the parameter $\epsilon$ tends to $0$, the perforated set $\Omega(\epsilon)$
degenerates to the punctured domain $\Omega^o \setminus \{0\}$.

As we have done in \cite{MuMi22}, for each $\epsilon \in ]0,\epsilon_0[$ we study a nonlinear boundary value problem for the Laplace operator: we consider a Neumann condition on $\partial \Omega^o$ and a nonlinear Robin condition on $\epsilon \partial \omega^i$. In order to define the boundary value problem in the set $\Omega(\epsilon)$, we fix two functions
\[
g^o\in C^{0,\alpha}(\partial \Omega^o)\, , \qquad g^i\in C^{0,\alpha}(\partial \omega^i)\, .
\]
Next we take a family $\{F_\epsilon\}_{\epsilon \in ]0,\epsilon_0[}$ of functions from $\mathbb{R}$ to $\mathbb{R}$, and two functions $\delta(\cdot)$ and $\rho(\cdot)$  from $]0,\epsilon_0[$ to $]0,+\infty[$.

Now for each $\epsilon \in ]0,\epsilon_0[$ we consider the following boundary value problem:
\begin{equation}
\label{bvpdelta}
\left\{
\begin{array}{ll}
\Delta u(x)=0 & \forall x \in \Omega(\epsilon)\,,\\
\frac{\partial}{\partial \nu_{\Omega^o}}u(x)=g^o(x) & \forall x \in \partial \Omega^o\, ,\\
\frac{\partial}{\partial \nu_{\epsilon\omega^i}}u(x)=\delta(\epsilon) F_\epsilon (u(x))+\frac{g^i(x/\epsilon)}{\rho(\epsilon)} & \forall x \in \epsilon \partial \omega^i\, ,
\end{array}
\right.
\end{equation}
where $\nu_{\Omega^o}$ and $\nu_{\epsilon\omega^i}$ denote the outward unit normal to $\partial \Omega^o$ and to $\partial (\epsilon \omega^i)$, respectively. 

As in in \cite{MuMi22}, our aim is to analyze the behavior of the solutions to problem \eqref{bvpdelta} as $\epsilon \to 0$ and to understand how the size of the hole and the functions $\delta$ and $\rho$ that intervene in the nonlinear Robin condition affect the asymptotic behavior of solutions to the problem \eqref{bvpdelta}. We will adapt the techniques of \cite{MuMi22} for the case of dimension $n\geq 3$ to the planar perforated domain of the present paper.

The paper is organized as follows. In Section \ref{model} we analyze a toy problem in an annular domain. In Section \ref{inteqfor} we transform problem \eqref{bvpdelta} into an equivalent system of integral equations. In Section \ref{rep}, we analyze such system and we prove our main results on the asymptotic behavior of a family of solutions and the corresponding energy integrals. Finally, Section \ref{case} contains some remarks on the linear case.

\section{A toy problem}\label{model}

As we have done in \cite{MuMi18, MuMi22}, we consider problem \eqref{bvpdelta} in the annular domain
\[
\Omega(\epsilon )\equiv \mathbb{B}_2(0,1)\setminus \overline{\mathbb{B}_2(0,\epsilon)}\, ,
\]
 where, for $r>0$, the symbol $\mathbb{B}_2(0,r)$ denotes the open ball in $\mathbb{R}^2$ of center $0$ and radius $r$. In other words, we take $\Omega^o\equiv \mathbb{B}_2(0,1)$ and $\omega^i\equiv \mathbb{B}_2(0,1)$.

We set $\epsilon_0=1$, $F_\epsilon(\tau)=\tau$ for all $\tau \in \mathbb{R}$ and for all $\epsilon \in ]0,\epsilon_0[$, $g^o=a$, and $g^i=b$, where $a,b \in \mathbb{R}$. In addition, we take  two functions $\delta, \rho \colon ]0,1[\mapsto ]0,+\infty[$ and for each $\epsilon \in ]0,1[$ we consider the problem

\begin{equation}
\label{bvpmodeldelta}
\left\{
\begin{array}{ll}
\Delta u(x)=0 & \forall x \in \mathbb{B}_2(0,1)\setminus \overline{\mathbb{B}_2(0,\epsilon)}\,,\\
\frac{\partial}{\partial \nu_{\mathbb{B}_2(0,1)}}u(x)=a & \forall x \in \partial \mathbb{B}_2(0,1)\, ,\\
\frac{\partial}{\partial \nu_{\mathbb{B}_2(0,\epsilon)}}u(x)=\delta(\epsilon) u(x)+\frac{b}{\rho(\epsilon)} & \forall x \in \partial \mathbb{B}_2(0,\epsilon)\, .
\end{array}
\right.
\end{equation}

It is well known that for each $\epsilon \in ]0,1[$ problem \eqref{bvpmodeldelta} has a unique solution in $C^{1,\alpha}(\overline{\Omega(\epsilon)})$. We denote such a solution by $u_\epsilon$.

On the other hand,  in the unperturbed domain $\mathbb{B}_2(0,1)$ the Neumann problem
\begin{equation}
\label{bvpmodeldelta0}
\left\{
\begin{array}{ll}
\Delta u(x)=0 & \forall x \in \mathbb{B}_2(0,1)\,,\\
\frac{\partial}{\partial \nu_{\mathbb{B}_2(0,1)}}u(x)=a & \forall x \in \partial \mathbb{B}_2(0,1)\, \\
\end{array}
\right.
\end{equation}
 is subject to compatibility conditions on the Neumann datum on $\partial \mathbb{B}_2(0,1)$. In particular, in this specific case of constant Neumann datum, problem \eqref{bvpmodeldelta0} has a solution if and only if
\begin{equation}\label{eq:compmodel:1}
a =0\, .
\end{equation}
For $a= 0$,  the Neumann problem \eqref{bvpmodeldelta0}  has the  one-dimensional space of constant functions in $\overline{\mathbb{B}_2(0,1)}$ as the space of solutions, whereas if instead we have that $a\neq 0$, then problem \eqref{bvpmodeldelta0} does not have any solution.

As a consequence, if the compatibility condition \eqref{eq:compmodel:1} does not hold, the unique solution $u_\epsilon$ of problem \eqref{bvpmodeldelta} clearly cannot converge to a solution of \eqref{bvpmodeldelta0} as $\epsilon \to 0$ (since problem \eqref{bvpmodeldelta0} has no solutions). Also, as we shall see, the solutions may diverge as $\epsilon \to 0$ even if $a=0$, because of the  terms $\delta(\epsilon)$ and $\rho(\epsilon)$. Here we wish to investigate how the Robin condition on the region $\epsilon \partial \omega^i$ influences  the asymptotic behavior of the solution as $\epsilon\to 0$.

Now, our goal is to explicitly construct  the solution $u_\epsilon$ of our toy problem \eqref{bvpmodeldelta} and then analyze the behavior of $u_\epsilon$ as $\epsilon \to 0$. We search for the solution $u_\epsilon$ in the form
\[
u_\epsilon(x)\equiv A_\epsilon \log |x| + B_\epsilon \qquad \forall x \in \overline{\Omega(\epsilon)}\, ,
\]
and we need to determine the constants $A_\epsilon$ and $B_\epsilon$ so that  the boundary conditions of problem \eqref{bvpmodeldelta} hold.

Since
\[
\nabla u_\epsilon (x)=A_\epsilon \frac{x}{|x|^2}\, ,
\]
to satisfy the Neumann condition on $\partial \mathbb{B}_2(0,1)$, we must have
\[
A_\epsilon=a\, .
\]
On the other hand, to satisfy  the Robin condition on $\partial \mathbb{B}_2(0,\epsilon)$, we need to determine $B_\epsilon$ so that
\begin{equation}
\label{cond_lin}
\frac{x}{|x|}\cdot a \frac{x}{|x|^2}=\delta(\epsilon) (a \log |x| +B_\epsilon)+\frac{b}{\rho(\epsilon)}\qquad \forall x\in\partial  \mathbb{B}_2(0,\epsilon)\, ,
\end{equation}
\textit{i.e.,}
\[
  \frac{a}{\epsilon}=\delta(\epsilon) (a \log \epsilon +B_\epsilon)+\frac{b}{\rho(\epsilon)}\qquad \forall x\in\partial  \mathbb{B}_2(0,\epsilon)\, .
\]
Therefore,
\[
B_\epsilon=\frac{1}{\delta(\epsilon)}\bigg(\frac{a}{\epsilon}-\frac{b}{\rho(\epsilon)}\bigg)-a\log \epsilon\, ,
\]
and, as a consequence, also
\begin{equation}\label{eq:sol:2}
u_\epsilon(x)\equiv a \log |x| +\frac{1}{\delta(\epsilon)}\bigg(\frac{a}{\epsilon}-\frac{b}{\rho(\epsilon)}\bigg)-a\log \epsilon \qquad \forall x \in \overline{\Omega(\epsilon)}\, .
\end{equation}
We can rewrite \eqref{eq:sol:2} as
\[
u_\epsilon(x)\equiv a \log |x| +\frac{1}{\epsilon\delta(\epsilon)}\bigg(a-b\frac{\epsilon}{\rho(\epsilon)}-a  \epsilon \delta(\epsilon)\log \epsilon\bigg) \qquad \forall x \in \overline{\Omega(\epsilon)}\, .
\]
This, for example, implies that if
\[
l_0\equiv \lim_{\epsilon\to 0}\epsilon\delta(\epsilon)\log \epsilon \in \mathbb{R} \, ,\qquad r_0\equiv \lim_{\epsilon\to 0}\frac{\epsilon}{\rho(\epsilon)} \in \mathbb{R}\, ,
\]
and
\[
a-b r_0-a  l_0\neq 0\, ,
\]
then the value of the solution 
$u_\epsilon(\overline{x})$ is asymptotic to $(a-b r_0-a  l_0)/(\epsilon \delta(\epsilon))$ as $\epsilon$ tends to $0$ for all fixed $\overline{x} \in \overline{\Omega}\setminus \{0\}$. 

This means that, under suitable assumptions on the behavior of $\delta(\epsilon)$ and $\rho(\epsilon)$ as $\epsilon \to 0$, the value of the solution $u_\epsilon(\overline{x})$ at a fixed point $\overline{x} \in \overline{\Omega}\setminus \{0\}$ behaves like $(a-b r_0-a  l_0)/(\epsilon  \delta(\epsilon))$. 
 If instead for each $\epsilon$ positive and small enough, we take $\tilde{x}_\epsilon$ such that $|\tilde{x}_\epsilon|=\epsilon$, then
\begin{align*}
u_\epsilon(\tilde{x}_\epsilon) &= a \log \epsilon +\frac{1}{\epsilon\delta(\epsilon)}\bigg(a-b\frac{\epsilon}{\rho(\epsilon)}-a  \epsilon \delta(\epsilon)\log \epsilon\bigg) \\
&= \frac{1}{\epsilon\delta(\epsilon)} \bigg( a \epsilon\delta(\epsilon) \log \epsilon +a-b\frac{\epsilon}{\rho(\epsilon)}-a  \epsilon \delta(\epsilon)\log \epsilon\bigg) \\
&= \frac{1}{\epsilon\delta(\epsilon)} \bigg( a-b\frac{\epsilon}{\rho(\epsilon)}\bigg) \, .
\end{align*}
In particular, if
\[
a-br_0\neq 0\, ,
\]
then the value $u_\epsilon(\tilde{x}_\epsilon)$ of the solution at $\tilde{x}_\epsilon$ is asymptotic to $( a-b r_0)/(\epsilon \delta(\epsilon))$ as $\epsilon \to 0$.

We now consider the energy integral of $u_\epsilon$.  A direct computation shows that
\begin{align*}
\int_{\Omega(\epsilon)}|\nabla u_\epsilon(x)|^2\, dx&=\int_{\Omega(\epsilon)}|\nabla \Big(a \log|x|\Big)|^2\, dx=\int_{\Omega(\epsilon)} a^2\frac{1}{|x|^2} \, dx\\
&\qquad =a^2 2 \pi \int_\epsilon^1 \frac{1}{r} \, dr=a^2 2 \pi\Big(-\log \epsilon\Big)\, .
\end{align*}

We note that equation \eqref{eq:sol:2} provides a solution of the linear toy problem \eqref{bvpmodeldelta} also if $\delta(\epsilon)<0$. In case $\delta(\epsilon)<0$, uniqueness for the solution  of  problem \eqref{bvpmodeldelta} may fail since indeed $\sigma=-\delta(\epsilon)$ could be a mixed Steklov-Neumann eigenvalue of problem
\[
\begin{cases}
\Delta u=0 & {\rm in\ }\Omega(\epsilon)\, ,\\
\frac{\partial}{\partial \nu_{\Omega(\epsilon)}}u=0 &{\rm on\ }\partial \mathbb{B}_2(0,1)\, ,\\
\frac{\partial}{\partial \nu_{\Omega(\epsilon)}}u=\sigma u & {\rm on\ }\partial \mathbb{B}_2(0,\epsilon)\, .
\end{cases}
\]
A detailed discussion on how to extends the results also to the case  $\delta(\epsilon)<0$ and the analysis of the behavior of Steklov-Neumann eigenvalues will be the subject of future investigations.

It is also interesting to look at a nonlinear toy problem with arbitrary functions $F_\epsilon(\cdot)$. Then repeating the same line of reasoning, the only difference appears in the equation \eqref{cond_lin}, that will take the form:
\begin{equation}
\label{cond_nonlin}
\frac{x}{|x|}\cdot a \frac{x}{|x|^2}=\delta(\epsilon) F_\epsilon(a \log |x| +B_\epsilon)+\frac{b}{\rho(\epsilon)}\qquad \forall x\in\partial  \mathbb{B}_2(0,\epsilon)\,.
\end{equation}
If we additionally assume that the functions $F_\epsilon: \mathbb{R} \to \mathbb{R}$  are invertible then by \eqref{cond_nonlin} for each $\epsilon$ the constant $B_\epsilon$ can be uniquely found:
 \[
  B_\epsilon=F_\epsilon^{-1}\left(\frac{a\rho(\epsilon)-\epsilon b}{\epsilon\rho(\epsilon)\delta(\epsilon)}\right)- a \log \epsilon\,  ,
\] 
and the analysis can be performed in a similar way as we have previously done. However, if the functions $F_\epsilon$ are not bijections, then the analysis of the existence (and possibly uniqueness) of the solution becomes more complex. For example, a solutions can be derived under specific conditions on the parameters if suitable rescaling of the functions $F_\epsilon$ are locally invertible. This shows how rich the problem is even in the simple situation of a circular annular domain. On the other hand, many of the features mentioned here are preserved for the general 2D case. Below, we provide an accurate analysis of the general problem formulated above, making, where appropriate, a reference to the similar feature highlighted here for the toy problem.

\section{An integral equation formulation of the boundary value problem}\label{inteqfor}

As in  \cite{MuMi18, MuMi22},  we use the Functional Analytic Approach introduced by Lanza de Cristoforis \cite{La02} to analyze problem \eqref{bvpdelta} when the parameter $\epsilon$ is close to $0$. We refer to  \cite{DaLaMu21} for a detailed presentation of the method. In order to apply such approach, we need to define classical objects of potential theory. We first denote by $S_2$ the fundamental solution of the Laplace operator, {\it i.e.} the function from $\mathbb{R}^2\setminus\{0\}$ to ${\mathbb{R}}$ defined by
\[
S_2(x)\equiv
\frac{1}{2\pi}\log |x| \qquad    \forall x\in
{\mathbb{R}}^{2}\setminus\{0\}\, .
\]
 By means of  $S_2$, we construct the single layer potentials, that we use to represent  the solutions of problem \eqref{bvpdelta}. So let $\Omega$ be a bounded open connected subset of $\mathbb{R}^2$ of class $C^{1,\alpha}$. We  introduce the single layer potential by 
\[
v[\partial\Omega,\mu](x)\equiv
\int_{\partial\Omega}S_2(x-y)\mu(y)\,d\sigma_{y}
\qquad\forall x\in \mathbb{R}^2\,,
\]
 for all $\mu\in C^{0}(\partial\Omega)$. If $\mu\in C^{0}(\partial{\Omega})$, then $v[\partial\Omega,\mu]$ is continuous in  $\mathbb{R}^2$. Moreover, if $\mu\in C^{0,\alpha}(\partial\Omega)$, then the function
$v^{+}[\partial\Omega,\mu]\equiv v[\partial\Omega,\mu]_{|\overline{\Omega}}$ belongs to $C^{1,\alpha}(\overline{\Omega})$, and the function
$v^{-}[\partial\Omega,\mu]\equiv v[\partial\Omega,\mu]_{|\mathbb{R}^2 \setminus \Omega}$ belongs to $C^{1,\alpha}_{\mathrm{loc}}
(\mathbb{R}^2 \setminus \Omega)$.  The normal derivative of the single layer potential on  $\partial \Omega$, instead, presents a jump. To describe such jump,  we set
\[
W^{\ast}[\partial\Omega,\mu](x)\equiv
\int_{\partial\Omega}\nu_{\Omega}(x) \cdot \nabla S_2(x-y)\mu(y)\,d\sigma_{y}
\qquad\forall x\in \partial \Omega\,,
\]
where $\nu_{\Omega}$  denotes the outward unit normal to $\partial \Omega$. If $\mu\in C^{0,\alpha}(\partial{\Omega})$, the function
$W^{\ast}[\partial\Omega,\mu]$ belongs to $C^{0,\alpha}(\partial \Omega)$ and we have
\[
\frac{\partial }{\partial \nu_{\Omega}}v^\pm[\partial \Omega,\mu]=\mp \frac{1}{2}\mu + W^{\ast}[\partial \Omega,\mu]\qquad \text{on $\partial \Omega$\, .}
\]

We will use density functions with zero integral mean and thus we find it convenient to set
\[
C^{0,\alpha}(\partial \Omega^o)_{0}\equiv
\bigg\{
f\in C^{0,\alpha}(\partial \Omega^o):\,\int_{\partial\Omega^o}f\,d\sigma=0
\bigg\}\,.
\]

By arguing as in \cite[\S 3]{MuMi22}, we are ready to establish in Proposition \ref{prop:corr} a correspondence between the solutions of problem \eqref{bvpdelta} and those of a (nonlinear) system of integral equations.

\begin{proposition}\label{prop:corr}
Let $\epsilon \in ]0,\epsilon_0[$. Then the map from the set of triples $(\mu^o,\mu^i,\xi)\in  C^{0,\alpha}(\partial\Omega^o)_{0}\times C^{0,\alpha}(\partial\omega^i)\times {\mathbb{R}}$ such that
\begin{align}
&-\frac{1}{2}\mu^o(x)+\int_{\partial \Omega^o}\nu_{\Omega^o}(x)\cdot \nabla S_2(x-y)\mu^o(y)\, d\sigma_y\nonumber\\
&\qquad+\int_{\partial \omega^i}\nu_{\Omega^o}(x)\cdot \nabla S_2(x-\epsilon s)\mu^i(s)\, d\sigma_s=g^o(x) \qquad \forall x \in \partial \Omega^o\, ,\label{eq:corr:1a}\\
&\frac{1}{2}\mu^i(t)+\epsilon\int_{\partial \Omega^o}\nu_{\omega^i}(t)\cdot \nabla S_2(\epsilon t-y)\mu^o(y)\, d\sigma_y+\int_{\partial \omega^i}\nu_{\omega^i}(t)\cdot \nabla S_2(t-s)\mu^i(s)\, d\sigma_s \nonumber\\
&\qquad=\epsilon \delta(\epsilon) F_\epsilon \Bigg(\int_{\partial \Omega^o}S_2(\epsilon t-y)\mu^o(y)\, d\sigma_y+ \int_{\partial \omega^i}S_2(t-s)\mu^i(s)\, d\sigma_s\nonumber\\&\quad\qquad+\frac{\log \epsilon}{2\pi}\int_{\partial \omega^i}\mu^i\, d\sigma+\frac{\xi}{\delta(\epsilon)\epsilon}\Bigg)+g^i(t)\frac{\epsilon}{\rho(\epsilon)} \qquad \forall t \in \partial \omega^i\, ,\label{eq:corr:1b}
\end{align}
to the set of those functions $u\in C^{1,\alpha}(\overline{\Omega(\epsilon)})$ which solve problem \eqref{bvpdelta},  which takes a triple $(\mu^o,\mu^i,\xi)$ to the function
\[
\int_{\partial \Omega^o}S_2(x-y)\mu^o(y)\, d\sigma_y+ \int_{\partial \omega^i}S_2(x-\epsilon s)\mu^i(s)\, d\sigma_s+\frac{\xi}{\delta(\epsilon)\epsilon}\qquad \forall x \in \overline{\Omega(\epsilon)}
\]
is a bijection.
\end{proposition}

By Proposition \ref{prop:corr} we can study the behavior of the solutions of boundary value problem \eqref{bvpdelta} by analyzing those of the system of integral equations \eqref{eq:corr:1a}-\eqref{eq:corr:1b} as $\epsilon \to 0$. As we have done in \cite{MuMi22}, we make some structural assumptions on the nonlinearity and we assume that
\begin{equation}\label{eq:addass:1}
\begin{split}
&\text{there exist $\epsilon_1 \in ]0,\epsilon_0[$, $m \in \mathbb{N}$, a real analytic function $\tilde{F}$ from $\mathbb{R}^{m+1}$ to $\mathbb{R}$,}\\
&\text{a function $\eta(\cdot)$ from $]0,\epsilon_1[$ to $\mathbb{R}^{m}$ such that $\eta_0\equiv \lim_{\epsilon\to 0}\eta(\epsilon) \in \mathbb{R}^m$ and that}\\
&\text{$\epsilon \delta(\epsilon) F_\epsilon \Big(\frac{1}{\epsilon \delta(\epsilon)}\tau\Big)=\tilde{F}(\tau,\eta(\epsilon))$ for all $(\tau,\epsilon) \in\mathbb{R}\times ]0,\epsilon_1[$.}
\end{split}
\end{equation}

\section{Analytic representation formulas for the solution of the boundary value problem}\label{rep}

Under the additional assumption \eqref{eq:addass:1}, we can rewrite the set of equations \eqref{eq:corr:1a}-\eqref{eq:corr:1b}  as
\begin{align}
&-\frac{1}{2}\mu^o(x)+\int_{\partial \Omega^o}\nu_{\Omega^o}(x)\cdot \nabla S_2(x-y)\mu^o(y)\, d\sigma_y\nonumber\\
&\qquad+\int_{\partial \omega^i}\nu_{\Omega^o}(x)\cdot \nabla S_2(x-\epsilon s)\mu^i(s)\, d\sigma_s=g^o(x) \qquad \forall x \in \partial \Omega^o\, ,\label{eq:corr:1a2}\\
&\frac{1}{2}\mu^i(t)+\epsilon\int_{\partial \Omega^o}\nu_{\omega^i}(t)\cdot \nabla S_2(\epsilon t-y)\mu^o(y)\, d\sigma_y+\int_{\partial \omega^i}\nu_{\omega^i}(t)\cdot \nabla S_2(t-s)\mu^i(s)\, d\sigma_s \nonumber\\
&\qquad=\tilde{F} \Bigg(\epsilon \delta(\epsilon)\int_{\partial \Omega^o}S_2(\epsilon t-y)\mu^o(y)\, d\sigma_y+\epsilon \delta(\epsilon) \int_{\partial \omega^i}S_2(t-s)\mu^i(s)\, d\sigma_s\nonumber\\
&\quad\qquad+\frac{\epsilon\delta(\epsilon)\log \epsilon}{2\pi} \int_{\partial \omega^i}\mu^i\, d\sigma+\xi,\eta(\epsilon)\Bigg)+g^i(t)\frac{\epsilon}{\rho(\epsilon)} \qquad \forall t \in \partial \omega^i\, ,\label{eq:corr:1b2}
\end{align}
for all $\epsilon \in ]0,\epsilon_1[$. In order to pass to the limit as $\epsilon \to 0$ in equations \eqref{eq:corr:1a2}-\eqref{eq:corr:1b2}, we need to know the asymptotic behavior for $\epsilon$ close to $0$ of the quantities $\epsilon\delta(\epsilon)$, $\epsilon\delta(\epsilon)\log \epsilon$, and $\frac{\epsilon}{\rho(\epsilon)}$ which appear in \eqref{eq:corr:1b2}. As a consequence, we now assume that
\begin{equation}\label{eq:Lmbd:limass}
l_0\equiv \lim_{\epsilon\to 0}\epsilon\delta(\epsilon)\log \epsilon \in \mathbb{R} \, ,\qquad r_0\equiv \lim_{\epsilon\to 0}\frac{\epsilon}{\rho(\epsilon)} \in \mathbb{R}\, .
\end{equation}
Condition \eqref{eq:Lmbd:limass} implies also
\[
\lim_{\epsilon\to 0}\epsilon\delta(\epsilon)=0\, .
\]
In \eqref{eq:corr:1a2}-\eqref{eq:corr:1b2}, we replace the quantities
\[
\epsilon \delta(\epsilon)\, ,\qquad \epsilon\delta(\epsilon)\log \epsilon\, , \qquad \eta(\epsilon)\, ,\qquad \frac{\epsilon}{\rho(\epsilon)}\, ,
\]
by the auxiliary variables
\[
\gamma_1,\qquad\gamma_2,\qquad\gamma_3,\qquad\gamma_4,
\]
respectively, and we  introduce the operator $\Lambda\equiv (\Lambda^o, \Lambda^i)$ from $]-\epsilon_1,\epsilon_1[\times \mathbb{R}^{m+3}\times C^{0,\alpha}(\partial\Omega^o)_0\times C^{0,\alpha}(\partial\omega^i)\times \mathbb{R}$ to $C^{0,\alpha}(\partial\Omega^o)\times C^{0,\alpha}(\partial\omega^i)$  by setting
\begin{align}
\Lambda^o&[\epsilon,\gamma_1,\gamma_2,\gamma_3,\gamma_4,\mu^o,\mu^i,\xi](x)\equiv-\frac{1}{2}\mu^o(x)+\int_{\partial \Omega^o}\nu_{\Omega^o}(x)\cdot \nabla S_2(x-y)\mu^o(y)\, d\sigma_y\nonumber\\
&\qquad+\int_{\partial \omega^i}\nu_{\Omega^o}(x)\cdot \nabla S_2(x-\epsilon s)\mu^i(s)\, d\sigma_s-g^o(x) \qquad \forall x \in \partial \Omega^o\, ,\label{eq:Lmbd:1a2}
\end{align}
\begin{align}
\Lambda^i&[\epsilon,\gamma_1,\gamma_2,\gamma_3,\gamma_4,\mu^o,\mu^i,\xi](t)\equiv\frac{1}{2}\mu^i(t)+\epsilon\int_{\partial \Omega^o}\nu_{\omega^i}(t)\cdot \nabla S_2(\epsilon t-y)\mu^o(y)\, d\sigma_y\nonumber\\
&\qquad+\int_{\partial \omega^i}\nu_{\omega^i}(t)\cdot \nabla S_2(t-s)\mu^i(s)\, d\sigma_s \nonumber\\
&\qquad-\tilde{F} \Bigg(\gamma_1\int_{\partial \Omega^o}S_2(\epsilon t-y)\mu^o(y)\, d\sigma_y+\gamma_1 \int_{\partial \omega^i}S_2(t-s)\mu^i(s)\, d\sigma_s\nonumber\\
&\qquad+\frac{\gamma_2}{2\pi} \int_{\partial \omega^i}\mu^i\, d\sigma+\xi,\gamma_3\Bigg)-g^i(t)\gamma_4 \qquad \forall t \in \partial \omega^i\, ,\label{eq:Lmbd:1b2}
\end{align}
for all $(\epsilon,\gamma_1,\gamma_2,\gamma_3,\gamma_4,\mu^o,\mu^i,\xi)\in ]-\epsilon_1,\epsilon_1[\times \mathbb{R}^{m+3}\times C^{0,\alpha}(\partial\Omega^o)_0\times C^{0,\alpha}(\partial\omega^i)\times \mathbb{R}$.  By definitions \eqref{eq:Lmbd:1a2}-\eqref{eq:Lmbd:1b2},  for $\epsilon \in ]0,\epsilon_1[$ the system of equations
\begin{align}
&\Lambda^o[\epsilon,\epsilon\delta(\epsilon),\epsilon\delta(\epsilon)\log \epsilon,\eta(\epsilon),\frac{\epsilon}{\rho(\epsilon)},\mu^o,\mu^i,\xi](x)=0\qquad \forall x \in \partial \Omega^o\, ,\label{eq:equiv:1a2}\\
&\Lambda^i[\epsilon,\epsilon\delta(\epsilon),\epsilon\delta(\epsilon)\log \epsilon,\eta(\epsilon),\frac{\epsilon}{\rho(\epsilon)},\mu^o,\mu^i,\xi](t)=0\qquad \forall t \in \partial \omega^i\, ,\label{eq:equiv:1b2}
\end{align}
is equivalent to the system of integral equations \eqref{eq:corr:1a2}-\eqref{eq:corr:1b2}. Letting $\epsilon \to 0$ in \eqref{eq:equiv:1a2}-\eqref{eq:equiv:1b2}, we obtain the equations
\begin{align}
&-\frac{1}{2}\mu^o(x)+\int_{\partial \Omega^o}\nu_{\Omega^o}(x)\cdot \nabla S_2(x-y)\mu^o(y)\, d\sigma_y\nonumber\\
&\qquad+\nu_{\Omega^o}(x)\cdot \nabla S_2(x)\int_{\partial \omega^i}\mu^i(s)\, d\sigma_s=g^o(x) \qquad \forall x \in \partial \Omega^o\, ,\label{eq:limsys:1a2}\\
&\frac{1}{2}\mu^i(t)+\int_{\partial \omega^i}\nu_{\omega^i}(t)\cdot \nabla S_2(t-s)\mu^i(s)\, d\sigma_s \nonumber\\
&\qquad=\tilde{F} \Bigg(\frac{l_0}{2\pi} \int_{\partial \omega^i}\mu^i\, d\sigma+\xi,\eta_0\Bigg)+g^i(t)r_0 \qquad \forall t \in \partial \omega^i\, .\label{eq:limsys:1b2}
\end{align}

For $\epsilon \in ]0,\epsilon_1[$ small enough, we would like to prove the existence of solutions $(\mu^o,\mu^i,\xi)$ to \eqref{eq:equiv:1a2}-\eqref{eq:equiv:1b2} around a solution of the limiting system  \eqref{eq:limsys:1a2}-\eqref{eq:limsys:1b2}. Therefore, we further assume that
\begin{equation}\label{eq:exsol:2}
\begin{split}
&\text{the system \eqref{eq:limsys:1a2}-\eqref{eq:limsys:1b2} in the unknown $(\mu^o,\mu^i,\xi)$ admits}\\
&\text{ a solution $(\tilde{\mu}^o,\tilde{\mu}^i,\tilde{\xi})$ in $C^{0,\alpha}(\partial \Omega^o)_0\times C^{0,\alpha}(\partial \omega^i) \times \mathbb{R}$.}
\end{split}
\end{equation}
We now note that if $(\tilde{\mu}^o,\tilde{\mu}^i,\tilde{\xi})$ is a solution of the system \eqref{eq:limsys:1a2}-\eqref{eq:limsys:1b2}, by integrating \eqref{eq:limsys:1a2} on $\partial \Omega^o$ and by the equalities
\[
\int_{\partial \Omega^o}\int_{\partial \Omega^o}\nu_{\Omega^o}(x)\cdot \nabla S_2(x-y)\tilde{\mu}^o(y)\, d\sigma_y\, d\sigma_x =\frac{1}{2}\int_{\partial \Omega^o}\tilde{\mu}^o(y)\, d\sigma_y\,
\]
(cf.~\cite[Lemma~6.11]{DaLaMu21}) and
\[
\int_{\partial \Omega^o}\nu_{\Omega^o}(x)\cdot \nabla S_2(x)\, d\sigma_x=1\,
\]
(cf.~\cite[Corollary~4.6]{DaLaMu21}), we must have
\[
\int_{\partial \omega^i}\tilde{\mu}^i(s)\, d\sigma_s=\int_{\partial \Omega^o} g^o(x)\, d\sigma_x\, .
\]
This implies that the triple $(\tilde{\mu}^o,\tilde{\mu}^i,\tilde{\xi})$ in $C^{0,\alpha}(\partial \Omega^o)_0\times C^{0,\alpha}(\partial \omega^i) \times \mathbb{R}$ is a solution of the system \eqref{eq:limsys:1a2bis}--\eqref{eq:limsys:1c2bis} below
\begin{align}
&-\frac{1}{2}\mu^o(x)+\int_{\partial \Omega^o}\nu_{\Omega^o}(x)\cdot \nabla S_2(x-y)\mu^o(y)\, d\sigma_y\nonumber\\
&\qquad=g^o(x)-\nu_{\Omega^o}(x)\cdot \nabla S_2(x)\int_{\partial \Omega^o} g^o(y)\, d\sigma_y \qquad \forall x \in \partial \Omega^o\, ,\label{eq:limsys:1a2bis}\\
&\frac{1}{2}\mu^i(t)+\int_{\partial \omega^i}\nu_{\omega^i}(t)\cdot \nabla S_2(t-s)\mu^i(s)\, d\sigma_s \nonumber\\
&\qquad=\tilde{F} \Bigg(\frac{l_0}{2\pi} \int_{\partial \Omega^o}g^o\, d\sigma+\xi,\eta_0\Bigg)+g^i(t)r_0 \qquad \forall t \in \partial \omega^i\, ,\label{eq:limsys:1b2bis}\\
&\int_{\partial \omega^i}\mu^i(s)\, d\sigma_s=\int_{\partial \Omega^o} g^o(x)\, d\sigma_x\, .\label{eq:limsys:1c2bis}
\end{align}

By \cite[Theorem 6.25]{DaLaMu21}, we deduce that there exists a unique solution $\tilde{\mu}^o$ in $C^{0,\alpha}(\partial \Omega^o)_0$ of \eqref{eq:limsys:1a2bis}. In other words, if there exists a solution $(\tilde{\mu}^o,\tilde{\mu}^i,\tilde{\xi})$  of the system \eqref{eq:limsys:1a2}-\eqref{eq:limsys:1b2}, then $\tilde{\mu}^o$ is determined as the unique solution in $C^{0,\alpha}(\partial \Omega^o)_0$ of \eqref{eq:limsys:1a2bis}.

In order to have a pair $(\tilde{\mu}^i,\tilde{\xi})$ in $C^{0,\alpha}(\partial \omega^i) \times \mathbb{R}$ solving \eqref{eq:limsys:1b2bis}-\eqref{eq:limsys:1c2bis}, we observe that if there exists $\tilde{\xi} \in \mathbb{R}$ such that
\begin{equation}\label{eq:limsys:1d2bis}
\int_{\partial \Omega^o} g^o(x)\, d\sigma_x=|\partial\omega^i|_1\tilde{F} \Bigg(\frac{l_0}{2\pi} \int_{\partial \Omega^o}g^o\, d\sigma+\tilde{\xi},\eta_0\Bigg)+\int_{\partial \omega^i}g^i(t)\, d\sigma_t r_0
\end{equation}
then \cite[Corollary 6.15]{DaLaMu21} implies the existence of a unique solution $\tilde{\mu}^i$ in $C^{0,\alpha}(\partial \omega^i)$ of 
\begin{align*}
&\frac{1}{2}\mu^i(t)+\int_{\partial \omega^i}\nu_{\omega^i}(t)\cdot \nabla S_2(t-s)\mu^i(s)\, d\sigma_s \\
&\qquad=\tilde{F} \Bigg(\frac{l_0}{2\pi} \int_{\partial \Omega^o}g^o\, d\sigma+\tilde{\xi},\eta_0\Bigg)+g^i(t)r_0 \qquad \forall t \in \partial \omega^i\, ,
\end{align*}
and such solution satisfies also
\[
\int_{\partial \omega^i}\tilde{\mu}^i(s)\, d\sigma_s=\int_{\partial \Omega^o} g^o(x)\, d\sigma_x\, .
\]
In other words this means that if there exists $\tilde{\xi} \in \mathbb{R}$ such that equation \eqref{eq:limsys:1d2bis} holds, then there exists a unique pair $(\tilde{\mu}^o,\tilde{\mu}^i)$ in $C^{0,\alpha}(\partial \Omega^o)_0\times C^{0,\alpha}(\partial \omega^i)$ such that the triple $(\tilde{\mu}^o,\tilde{\mu}^i,\tilde{\xi})$ in $C^{0,\alpha}(\partial \Omega^o)_0\times C^{0,\alpha}(\partial \omega^i) \times \mathbb{R}$ solves system \eqref{eq:limsys:1a2}-\eqref{eq:limsys:1b2}.

Furthermore, we note that equality \eqref{eq:limsys:1d2bis} can be rewritten as 
\begin{equation}\label{eq:limsys:1e2bis}
\tilde{F} \Bigg(\frac{l_0}{2\pi} \int_{\partial \Omega^o}g^o\, d\sigma+\tilde{\xi},\eta_0\Bigg)=\frac{1}{|\partial\omega^i|_1}\bigg(r_0\int_{\partial \omega^i}g^i\, d\sigma - \int_{\partial \Omega^o} g^o\, d\sigma\bigg)\, .
\end{equation}
Thus, in particular, if $\tilde{F} (\cdot,\eta_0)$ is not globally invertible, there can be multiple $\tilde{\xi}\in \mathbb{R}$ such that \eqref{eq:limsys:1e2bis} holds.

In the following proposition, we study the solvability of the system of integral equations \eqref{eq:corr:1a2}-\eqref{eq:corr:1b2}, by applying the Implicit Function Theorem to $\Lambda$, under suitable assumptions on  the partial derivative $\partial_\tau \tilde{F} \Bigg(\frac{l_0}{2\pi} \int_{\partial \Omega^o} g^o\, d\sigma+\tilde{\xi},\eta_0\Bigg)$. The symbol $\partial_\tau \tilde{F}(\tau,\eta)$ denotes the partial derivative of $\tilde{F}$ with respect to the first variable.

\begin{proposition}\label{prop:Lmbd2}
Let assumptions \eqref{eq:addass:1} and  \eqref{eq:Lmbd:limass} hold. Let $(\tilde{\mu}^o,\tilde{\mu}^i,\tilde{\xi})$ be as in assumption  \eqref{eq:exsol:2}. Assume that
\[
\partial_\tau \tilde{F} \Bigg(\frac{l_0}{2\pi} \int_{\partial \Omega^o} g^o\, d\sigma+\tilde{\xi},\eta_0\Bigg) \neq 0\, .
\]
Then there exist $\epsilon_2 \in ]0,\epsilon_1[$, an open neighborhood $\mathcal{U}$ of $(0,l_0,\eta_0,r_0)$ in $\mathbb{R}^{m+3}$, an open neighborhood $\mathcal{V}$ of $(\tilde{\mu}^o,\tilde{\mu}^i,\tilde{\xi})$ in $C^{0,\alpha}(\partial \Omega^o)_0\times C^{0,\alpha}(\partial \omega^i) \times \mathbb{R}$, and a real analytic map $(M^o,M^i, \Xi)$ from $]-\epsilon_2,\epsilon_2[\times \mathcal{U}$ to $\mathcal{V}$ such that
\[
\Bigg(\epsilon\delta(\epsilon),\epsilon\delta(\epsilon)\log \epsilon,\eta(\epsilon),\frac{\epsilon}{\rho(\epsilon)}\Bigg) \in \mathcal{U}\qquad \forall \epsilon \in ]0,\epsilon_2[\, ,
\]
and such that the set of zeros of $\Lambda$ in
$]-\epsilon_2,\epsilon_2[\times \mathcal{U}\times \mathcal{V}$ coincides with the graph of $(M^o,M^i,\Xi)$. In particular,
\[
\Bigg(M^o[0,0,l_0,\eta_0,r_0],M^i[0,0,l_0,\eta_0,r_0],\Xi[0,0,l_0,\eta_0,r_0]\Bigg)=(\tilde{\mu}^o,\tilde{\mu}^i,\tilde{\xi})\,.
\]
\end{proposition}
\begin{proof}
Standard results of classical potential theory (see, {\it e.g.}, \cite{DaLaMu21}, Miranda \cite{Mi65}, Lanza de Cristoforis and Rossi \cite{LaRo04}),  real analyticity results for integral operators with real analytic kernel (\cite{LaMu13}),  assumption \eqref{eq:addass:1} and real analyticity results for the composition operator (\cite[p.~10]{BoTo73}, \cite{He82}, and Valent \cite[Thm.~5.2]{Va88}) imply that $\Lambda$ is a real analytic operator from $]-\epsilon_1,\epsilon_1[\times \mathbb{R}^{m+3}\times C^{0,\alpha}(\partial\Omega^o)_0\times C^{0,\alpha}(\partial\omega^i)\times \mathbb{R}$ to $C^{0,\alpha}(\partial\Omega^o)\times C^{0,\alpha}(\partial\omega^i)$. We verify that by standard calculus in Banach space the partial differential $\partial_{(\mu^o,\mu^i,\xi)}\Lambda[0,0,l_0,\eta_0,r_0,\tilde{\mu}^o,\tilde{\mu}^i,\tilde{\xi}]$ of $\Lambda$ at $(0,0,l_0,\eta_0,r_0,\tilde{\mu}^o,\tilde{\mu}^i,\tilde{\xi})$ with respect to the variable $(\mu^o,\mu^i,\xi)$ is delivered by
\begin{align}
\partial_{(\mu^o,\mu^i,\xi)}&\Lambda^o[0,0,l_0,\eta_0,r_0,\tilde{\mu}^o,\tilde{\mu}^i,\tilde{\xi}](\overline{\mu}^o,\overline{\mu}^i,\overline{\xi})(x)\nonumber\\&\equiv-\frac{1}{2}\overline{\mu}^o(x)+\int_{\partial \Omega^o}\nu_{\Omega^o}(x)\cdot \nabla S_2(x-y)\overline{\mu}^o(y)\, d\sigma_y\nonumber\\
&\qquad+\nu_{\Omega^o}(x)\cdot \nabla S_2(x)\int_{\partial \omega^i}\overline{\mu}^i(s)\, d\sigma_s \qquad \forall x \in \partial \Omega^o\, ,\nonumber
\\
\partial_{(\mu^o,\mu^i,\xi)}&\Lambda^i[0,0,l_0,\eta_0,r_0,\tilde{\mu}^o,\tilde{\mu}^i,\tilde{\xi}](\overline{\mu}^o,\overline{\mu}^i,\overline{\xi})(t)\nonumber\\&\equiv\frac{1}{2}\overline{\mu}^i(t)+\int_{\partial \omega^i}\nu_{\omega^i}(t)\cdot \nabla S_2(t-s)\overline{\mu}^i(s)\, d\sigma_s \nonumber\\
&-\partial_\tau \tilde{F} \Bigg(\frac{l_0}{2\pi}\int_{\partial \Omega^o} g^o\, d\sigma+\tilde{\xi},\eta_0\Bigg) \Bigg(\frac{l_0}{2\pi} \int_{\partial \omega^i}\overline{\mu}^i\, d\sigma+\overline{\xi}\Bigg)\qquad \forall t \in \partial \omega^i\, ,\nonumber
\end{align}
for all $(\overline{\mu}^o,\overline{\mu}^i,\overline{\xi}) \in C^{0,\alpha}(\partial\Omega^o)_0\times C^{0,\alpha}(\partial\omega^i)\times \mathbb{R}$. The next step is to prove that the partial differential $\partial_{(\mu^o,\mu^i,\xi)}\Lambda[0,0,l_0,\eta_0,r_0,\tilde{\mu}^o,\tilde{\mu}^i,\tilde{\xi}]$ is a homeomorphism from $C^{0,\alpha}(\partial\Omega^o)_0\times C^{0,\alpha}(\partial\omega^i)\times \mathbb{R}$ onto $C^{0,\alpha}(\partial\Omega^o) \times C^{0,\alpha}(\partial\omega^i)$.  We observe that the partial differential $\partial_{(\mu^o,\mu^i,\xi)}\Lambda[0,0,l_0,\eta_0,r_0,\tilde{\mu}^o,\tilde{\mu}^i,\tilde{\xi}]$ is a Fredholm operator of index $0$: indeed it is the sum of an invertible operator and a compact operator. As a consequence, to prove that the operator $\partial_{(\mu^o,\mu^i,\xi)}\Lambda[0,0,l_0,\eta_0,r_0,\tilde{\mu}^o,\tilde{\mu}^i,\tilde{\xi}]$ is homeomorphism, it is enough to show that it is injective. Therefore, let us assume that
\[
\partial_{(\mu^o,\mu^i,\xi)}\Lambda[0,0,l_0,\eta_0,r_0,\tilde{\mu}^o,\tilde{\mu}^i,\tilde{\xi}](\overline{\mu}^o,\overline{\mu}^i,\overline{\xi})=0\, .
\]
We integrate equality
\[
\partial_{(\mu^o,\mu^i,\xi)}\Lambda^o[0,0,l_0,\eta_0,r_0,\tilde{\mu}^o,\tilde{\mu}^i,\tilde{\xi}](\overline{\mu}^o,\overline{\mu}^i,\overline{\xi})(x)=0 \qquad \forall x \in \partial \Omega^o\, ,
\]
that together with the equalities
\[
\int_{\partial \Omega^o}\int_{\partial \Omega^o}\nu_{\Omega^o}(x)\cdot \nabla S_2(x-y)\overline{\mu}^o(y)\, d\sigma_y\, d\sigma_x =\frac{1}{2}\int_{\partial \Omega^o}\overline{\mu}^o(y)\, d\sigma_y\,
\]
(cf.~\cite[Lemma~6.11]{DaLaMu21}) and
\[
\int_{\partial \Omega^o}\nu_{\Omega^o}(x)\cdot \nabla S_2(x)\, d\sigma_x=1\,
\]
(cf.~\cite[Corollary~4.6]{DaLaMu21}), implies
\begin{equation}\label{eq:diffLmbd:int2}
\int_{\partial \omega^i}\overline{\mu}^i(s)\, d\sigma_s=0\, .
\end{equation}
Accordingly,
\[
-\frac{1}{2}\overline{\mu}^o(x)+\int_{\partial \Omega^o}\nu_{\Omega^o}(x)\cdot \nabla S_2(x-y)\overline{\mu}^o(y)\, d\sigma_y=0 \qquad \forall x \in \partial \Omega^o\, .
\]
Since $\int_{\partial \Omega^o}\overline{\mu}^o\, d\sigma=0$, by \cite[Theorem 6.25]{DaLaMu21}  we have $\overline{\mu}^o=0$. Then we note that by \eqref{eq:diffLmbd:int2} equality
\[
\partial_{(\mu^o,\mu^i,\xi)}\Lambda^i[0,0,l_0,\eta_0,r_0,\tilde{\mu}^o,\tilde{\mu}^i,\tilde{\xi}](\overline{\mu}^o,\overline{\mu}^i,\overline{\xi})(t)=0 \qquad \forall t \in \partial \omega^i\,
\]
reads as
\begin{equation}\label{eq:diffLmbd:1d2}
\begin{split}
\frac{1}{2}\overline{\mu}^i(t)+\int_{\partial \omega^i}\nu_{\omega^i}(t)\cdot \nabla S_2(t-s)\overline{\mu}^i(s)\, d\sigma_s-\partial_\tau \tilde{F} \Bigg(\frac{l_0}{2\pi}\int_{\partial \Omega^o} g^o\, d\sigma+\tilde{\xi},\eta_0\Bigg) \overline{\xi}=0 \quad \forall t \in \partial \omega^i\, .
\end{split}
\end{equation}
By equality \eqref{eq:diffLmbd:int2}, by
\[
\int_{\partial \omega^i}\int_{\partial \omega^i}\nu_{\omega^i}(t)\cdot \nabla S_2(t-s)\overline{\mu}^i(s)\, d\sigma_s\, d\sigma_t =\frac{1}{2}\int_{\partial \omega^i}\overline{\mu}^i(s)\, d\sigma_s\,
\]
(cf.~\cite[Lemma~6.11]{DaLaMu21}), and by integrating \eqref{eq:diffLmbd:1d2} on $\partial \omega^i$, we deduce that $\overline{\xi}=0$. Then  \cite[Corollary 6.15]{DaLaMu21} implies that $\overline{\mu}^i=0$. Hence, we have shown that the operator $\partial_{(\mu^o,\mu^i,\xi)}\Lambda[0,0,l_0,\eta_0,r_0,\tilde{\mu}^o,\tilde{\mu}^i,\tilde{\xi}]$ is injective, and as a consequence, being a Fredholm operator of index $0$, also a homeomorphism. Therefore, we can apply  the Implicit Function Theorem for real analytic maps in Banach spaces (cf.~Deimling \cite[Thm.~15.3]{De85}) and deduce that there exist $\epsilon_2 \in ]0,\epsilon_1[$, an open neighborhood $\mathcal{U}$ of $(0,l_0,\eta_0,r_0)$ in $\mathbb{R}^{m+3}$, an open neighborhood $\mathcal{V}$ of $(\tilde{\mu}^o,\tilde{\mu}^i,\tilde{\xi})$ in $C^{0,\alpha}(\partial \Omega^o)_0\times C^{0,\alpha}(\partial \omega^i) \times \mathbb{R}$, and a real analytic map $(M^o,M^i, \Xi)$ from $]-\epsilon_2,\epsilon_2[\times \mathcal{U}$ to $\mathcal{V}$ such that
\[
\Bigg(\epsilon\delta(\epsilon),\epsilon\delta(\epsilon)\log \epsilon,\eta(\epsilon),\frac{\epsilon}{\rho(\epsilon)}\Bigg) \in \mathcal{U}\qquad \forall \epsilon \in ]0,\epsilon_2[\, ,
\]
and such that the set of zeros of $\Lambda$ in
$]-\epsilon_2,\epsilon_2[\times \mathcal{U}\times \mathcal{V}$ coincides with the graph of $(M^o,M^i,\Xi)$. In particular,
\[
\Bigg(M^o[0,0,l_0,\eta_0,r_0],M^i[0,0,l_0,\eta_0,r_0],\Xi[0,0,l_0,\eta_0,r_0]\Bigg)=(\tilde{\mu}^o,\tilde{\mu}^i,\tilde{\xi})\,,
\]
and thus the proof is complete.
\end{proof}

By Proposition \ref{prop:Lmbd2} we know that there exists a family of solutions of the system of integral equations \eqref{eq:corr:1a2}-\eqref{eq:corr:1b2}. Then we can exploit the representation formula of Proposition \ref{prop:corr}  and introduce a family of solutions to \eqref{bvpdelta}. We do so in the following Definition \ref{def:udelta2}.

\begin{definition}\label{def:udelta2}
Let the assumptions of Proposition \ref{prop:Lmbd2} hold. Then we set
\[
\begin{split}
u(\epsilon,x)=&\int_{\partial \Omega^o}S_2(x-y)M^o[\epsilon,\epsilon\delta(\epsilon),\epsilon\delta(\epsilon)\log \epsilon, \eta(\epsilon),\frac{\epsilon}{\rho(\epsilon)}](y)\, d\sigma_y\\
&+ \int_{\partial \omega^i}S_2(x-\epsilon s)M^i[\epsilon,\epsilon\delta(\epsilon),\epsilon\delta(\epsilon)\log \epsilon, \eta(\epsilon),\frac{\epsilon}{\rho(\epsilon)}](s)\, d\sigma_s\\
&+\frac{\Xi[\epsilon,\epsilon\delta(\epsilon),\epsilon\delta(\epsilon)\log \epsilon, \eta(\epsilon),\frac{\epsilon}{\rho(\epsilon)}]}{\delta(\epsilon)\epsilon}\qquad \forall x \in \overline{\Omega(\epsilon)}\, ,
\end{split}
\]
for all $\epsilon \in ]0,\epsilon_2[$.
\end{definition}

We  are ready to exploit the representation formula of Proposition \ref{prop:corr} and  the analyticity result of  Proposition \ref{prop:Lmbd2} concerning the solutions of the system of integral equations \eqref{eq:corr:1a2}-\eqref{eq:corr:1b2}  in order to prove formulas for  suitable restrictions of the solutions $u(\epsilon,\cdot)$  and for the corresponding energy integral in terms of real analytic maps.

We start by considering the restriction of the solution $u(\epsilon,\cdot)$ to a set which is ``far'' from the point where the hole degenerates.

\begin{theorem}\label{thm:rep2}
Let the assumptions of Proposition \ref{prop:Lmbd2} hold.  Let $\Omega_M$ be a bounded open subset of $\Omega^o$ such that $0 \not \in \overline{\Omega_M}$. Then there exist $\epsilon_M \in ]0,\epsilon_2[$ and a real analytic map $U_M$ from $]-\epsilon_M,\epsilon_M[\times \mathcal{U}$ to  $C^{1,\alpha}(\overline{ \Omega_M})$ such that
\[
\overline{\Omega_M}\subseteq \overline{\Omega(\epsilon)}\qquad \forall \epsilon \in ]0,\epsilon_M[\, ,
\]
and that
\begin{equation}\label{eq:rep2:a}
\begin{split}
u(\epsilon,x)= U_M[\epsilon,\epsilon\delta(\epsilon),\epsilon\delta(\epsilon)\log \epsilon, \eta(\epsilon),\frac{\epsilon}{\rho(\epsilon)}](x)+&\frac{\Xi[\epsilon,\epsilon\delta(\epsilon),\epsilon\delta(\epsilon)\log \epsilon, \eta(\epsilon),\frac{\epsilon}{\rho(\epsilon)}]}{\delta(\epsilon)\epsilon}\\&\qquad\qquad\qquad\qquad\forall x\in \overline{\Omega_M}\,,
\end{split}
\end{equation}
for all $\epsilon \in ]0,\epsilon_M[$. Moreover, if we set
\[
\begin{split}
\tilde{u}_{M}(x)\equiv  &\int_{\partial \Omega^o}S_2(x-y)\tilde{\mu}^o(y)\, d\sigma_y\qquad \forall x \in \overline{\Omega^o}\, ,
\end{split}
\]
we  have that $U_M[0,0,l_0,\eta_0,r_0]=\tilde{u}_{M |\overline{\Omega_M}}+S_{2|\overline{\Omega_M}}\int_{\partial \Omega^o}g^o\, d\sigma$, and $\tilde{u}_M$ is a solution of the Neumann problem
\begin{equation}
\label{eq:rep2:1}
\left\{
\begin{array}{ll}
\Delta u(x)=0 & \forall x \in \Omega^o\,,\\
\frac{\partial}{\partial \nu_{\Omega^o}}u(x)=g^o(x)-\frac{\partial}{\partial \nu_{\Omega^o}}S_{2}(x)\int_{\partial \Omega^o}g^o\, d\sigma & \forall x \in \partial \Omega^o\, .
\end{array}
\right.
\end{equation}
\end{theorem}
\begin{proof}
We can take $\epsilon_M \in ]0,\epsilon_2[$ small enough so that
\[
\overline{\Omega_M}\cap \epsilon\overline{\omega^i}=\emptyset\qquad  \forall  \epsilon \in ]-\epsilon_M,\epsilon_M[\, .
\]
Recalling  Definition \ref{def:udelta2}, we set
\[
\begin{split}
U_M[\epsilon,\gamma_1,\gamma_2, \gamma_3,\gamma_4](x)\equiv &\int_{\partial \Omega^o}S_2(x-y)M^o[\epsilon,\gamma_1,\gamma_2, \gamma_3,\gamma_4](y)\, d\sigma_y\\
&+ \int_{\partial \omega^i}S_2(x-\epsilon s)M^i[\epsilon,\gamma_1,\gamma_2, \gamma_3,\gamma_4](s)\, d\sigma_s\qquad \forall x \in \overline{\Omega_M}\, ,
\end{split}
\]
for all $(\epsilon,\gamma_1,\gamma_2, \gamma_3,\gamma_4) \in ]-\epsilon_M,\epsilon_M[\times \mathcal{U}$. Then Proposition \ref{prop:Lmbd2} and real analyticity results for integral operators with real analytic kernel (cf.~\cite{LaMu13}) imply that $U_M$ is a real analytic map from $]-\epsilon_M,\epsilon_M[\times \mathcal{U}$ to  $C^{1,\alpha}(\overline{ \Omega_M})$ and that equality \eqref{eq:rep2:a} holds. By Proposition \ref{prop:Lmbd2}, we also have $U_M[0,0,l_0,\eta_0,r_0]=\tilde{u}_{M|\overline{\Omega_M}}+S_{2|\overline{\Omega_M}}\int_{\partial \Omega^o}g^o\, d\sigma$, Moreover,  standard properties of the single layer potential (cf.~\cite[\S 4.4]{DaLaMu21}) imply that $\tilde{u}_M$ is a solution of the Neumann problem \eqref{eq:rep2:1}. 
\end{proof}

Then we consider the behavior of the rescaled solution $u(\epsilon,\epsilon t)$.

\begin{theorem}\label{thm:rep2micro}
Let the assumptions of Proposition \ref{prop:Lmbd2} hold. {Let $Z_m$ be the real analytic map from $ ]-\epsilon_2,\epsilon_2[\times \mathcal{U}$ to $\mathbb{R}$ defined by
\[
Z_m[\epsilon,\gamma_1,\gamma_2, \gamma_3,\gamma_4](t)\equiv   \int_{\partial \omega^i}M^i[\epsilon,\gamma_1,\gamma_2, \gamma_3,\gamma_4](s)\, d\sigma_s \, ,
\]
for all $(\epsilon,\gamma_1,\gamma_2, \gamma_3,\gamma_4) \in ]-\epsilon_2,\epsilon_2[\times \mathcal{U}$.} Let $\Omega_m$ be a bounded open subset of $\mathbb{R}^2\setminus \overline{\omega^i}$. Then there exist $\epsilon_m \in ]0,\epsilon_2[$ and a real analytic map $U_m$ from $]-\epsilon_m,\epsilon_m[\times \mathcal{U}$ to  $C^{1,\alpha}(\overline{ \Omega_m})$ such that
\[
\epsilon\overline{\Omega_m}\subseteq \overline{\Omega(\epsilon)}\qquad \forall \epsilon \in ]0,\epsilon_m[\, ,
\]
and that
\[
\begin{split}
u(\epsilon,\epsilon t)= U_m[\epsilon,\epsilon\delta(\epsilon),\epsilon\delta(\epsilon)\log \epsilon,\eta(\epsilon),\frac{\epsilon}{\rho(\epsilon)}&](t)+\frac{\log \epsilon}{2\pi}Z_m[\epsilon,\epsilon\delta(\epsilon),\epsilon\delta(\epsilon)\log \epsilon,\eta(\epsilon),\frac{\epsilon}{\rho(\epsilon)}]\\
+&\frac{\Xi[\epsilon,\epsilon\delta(\epsilon),\epsilon\delta(\epsilon)\log \epsilon,\eta(\epsilon),\frac{\epsilon}{\rho(\epsilon)}]}{\delta(\epsilon)\epsilon} \qquad\forall t\in \overline{\Omega_m}\,,
\end{split}
\]
for all $\epsilon \in ]0,\epsilon_m[$. Moreover, if we set
\[
\begin{split}
\tilde{u}_{m}(t)\equiv  &\int_{\partial \omega^i}S_2(t-s)\tilde{\mu}^i(s)\, d\sigma_s+  \int_{\partial \Omega^o}S_2(y)\tilde{\mu}^o(y)\, d\sigma_y \qquad \forall t \in \mathbb{R}^2 \setminus \omega^i\, ,
\end{split}
\]
we  have that $U_m[0,0,l_0,\eta_0,r_0]=\tilde{u}_{m |\overline{\Omega_m}}$, and $\tilde{u}_m$ is a solution of the Neumann problem
\begin{equation}
\begin{split}
\label{eq:rep2micro:1}
\left\{
\begin{array}{ll}
\Delta u(t)=0 & \forall t \in\mathbb{R}^2 \setminus \overline{\omega^i}\,,\\
\frac{\partial}{\partial \nu_{\omega^i}}u(t)=\tilde{F} \Bigg(\frac{l_0}{2\pi} \int_{\partial \Omega^o}g^o\, d\sigma+\tilde{\xi},\eta_0\Bigg)+g^i(t)r_0& \forall t \in \partial \omega^i\, ,
\end{array}
\right.\end{split}
\end{equation}
and
\[
Z_m[0,0,l_0,\eta_0,r_0]=\int_{\partial \Omega^o}g^o\, d\sigma\, .
\]
\end{theorem}
\begin{proof}
We take $\epsilon_m \in ]0,\epsilon_2[$ small enough and we can assume that
\[
\epsilon\overline{\Omega_m}\subseteq \overline{\Omega^o}\qquad  \forall  \epsilon \in ]-\epsilon_m,\epsilon_m[\, .
\]
If $\epsilon \in ]0,\epsilon_m[$ then
\[
\begin{split}
u(\epsilon,\epsilon t)=&\int_{\partial \Omega^o}S_2(\epsilon t-y)M^o[\epsilon,\epsilon\delta(\epsilon),\epsilon\delta(\epsilon)\log \epsilon, \eta(\epsilon),\frac{\epsilon}{\rho(\epsilon)}](y)\, d\sigma_y\\
&+ \int_{\partial \omega^i}S_2(\epsilon t-\epsilon s)M^i[\epsilon,\epsilon\delta(\epsilon),\epsilon\delta(\epsilon)\log \epsilon, \eta(\epsilon),\frac{\epsilon}{\rho(\epsilon)}](s)\, d\sigma_s\\
&+\frac{\Xi[\epsilon,\epsilon\delta(\epsilon),\epsilon\delta(\epsilon)\log \epsilon, \eta(\epsilon),\frac{\epsilon}{\rho(\epsilon)}]}{\delta(\epsilon)\epsilon}\\
=&\int_{\partial \Omega^o}S_2(\epsilon t-y)M^o[\epsilon,\epsilon\delta(\epsilon),\epsilon\delta(\epsilon)\log \epsilon, \eta(\epsilon),\frac{\epsilon}{\rho(\epsilon)}](y)\, d\sigma_y\\
&+\int_{\partial \omega^i}S_2( t- s)M^i[\epsilon,\epsilon\delta(\epsilon),\epsilon\delta(\epsilon)\log \epsilon, \eta(\epsilon),\frac{\epsilon}{\rho(\epsilon)}](s)\, d\sigma_s\\
&+\frac{\log \epsilon}{2\pi}\int_{\partial \omega^i}M^i[\epsilon,\epsilon\delta(\epsilon),\epsilon\delta(\epsilon)\log \epsilon, \eta(\epsilon),\frac{\epsilon}{\rho(\epsilon)}](s)\, d\sigma_s\\
&+\frac{\Xi[\epsilon,\epsilon\delta(\epsilon),\epsilon\delta(\epsilon)\log \epsilon, \eta(\epsilon),\frac{\epsilon}{\rho(\epsilon)}]}{\delta(\epsilon)\epsilon}\qquad \forall t \in \overline{\Omega_m}\, 
\end{split}
\]
(cf.~Definition \ref{def:udelta2}). Hence, we set
\[
\begin{split}
U_m[\epsilon,\gamma_1,\gamma_2, \gamma_3,\gamma_4](t)\equiv & \int_{\partial \Omega^o}S_2(\epsilon t-y)M^o[\epsilon,\gamma_1,\gamma_2, \gamma_3,\gamma_4](y)\, d\sigma_y\\
&+ \int_{\partial \omega^i}S_2(t-s)M^i[\epsilon,\gamma_1,\gamma_2, \gamma_3,\gamma_4](s)\, d\sigma_s\qquad \forall t \in \overline{\Omega_m}\, ,
\end{split}
\]
for all $(\epsilon,\gamma_1,\gamma_2, \gamma_3,\gamma_4) \in ]-\epsilon_m,\epsilon_m[\times \mathcal{U}$. By arguing as in the proof of Theorem \ref{thm:rep2}, we verify that $U_m$ and the map $Z_m$ of the statement are real analytic from $]-\epsilon_m,\epsilon_m[\times \mathcal{U}$ to  $C^{1,\alpha}(\overline{ \Omega_m})$ and to $\mathbb{R}$, respectively, and that equality \eqref{eq:rep2:a} holds. By Proposition \ref{prop:Lmbd2}, we also deduce that $Z_m[0,0,l_0,\eta_0,r_0]=\int_{\partial \Omega^o}g^o\, d\sigma$, that  $U_m[0,0,l_0,\eta_0,r_0]=\tilde{u}_{m|\overline{\Omega_m}}$. Also by standard properties of the single layer potential (cf.~\cite[\S 4.4]{DaLaMu21}), we deduce that $\tilde{u}_m$ is a solution of the Neumann problem \eqref{eq:rep2micro:1}. 
\end{proof}

\begin{remark}
We note that if $\int_{\partial \omega^i}\tilde{\mu}^i\, d\sigma\neq 0$ ({\it i.e.}, if $\int_{\partial \Omega^o}g^o\, d\sigma\neq 0$), then the function $\tilde{u}_{m}$ of Theorem \ref{thm:rep2micro} is not harmonic at infinity (cf.~\cite[Definition 3.21 and Theorem 4.23]{DaLaMu21}).
\end{remark}

Finally, we study the behavior of the energy integral $\int_{\Omega(\epsilon)}|\nabla u(\epsilon,x)|^2\, dx$ as the parameter $\epsilon$ approaches $0$.

\begin{theorem}\label{thm:enrep2}
Let the assumptions of Proposition \ref{prop:Lmbd2} hold. Let $\tilde{u}_{M}$ and $\tilde{u}_{m}$ be as in Theorem \ref{thm:rep2} and Theorem \ref{thm:rep2micro}, respectively. Then there exist $\epsilon_e \in ]0,\epsilon_2[$ and  two real analytic maps $E_1$ and $E_2$ from $]-\epsilon_e,\epsilon_e[\times \mathcal{U}$ to $\mathbb{R}$ such that
\begin{equation}\label{eq:rep2:3}
\begin{split}
\int_{\Omega(\epsilon)}|\nabla u(\epsilon,x)|^2\, dx &= E_1[\epsilon,\epsilon\delta(\epsilon),\epsilon\delta(\epsilon)\log \epsilon, \eta(\epsilon),\frac{\epsilon}{\rho(\epsilon)}]+ (\log\epsilon )E_2[\epsilon,\epsilon\delta(\epsilon),\epsilon\delta(\epsilon)\log \epsilon, \eta(\epsilon),\frac{\epsilon}{\rho(\epsilon)}]\,,
\end{split}
\end{equation}
for all $\epsilon \in ]0,\epsilon_e[$. Moreover,
\begin{equation}\label{eq:rep2:4}
\begin{split}
E_1[0,0,l_0,\eta_0,r_0]&= \int_{\partial \Omega^o} \Big(\tilde{u}_M(x)+S_2(x)\int_{\partial \Omega^o}g^o\,d\sigma\Big) \nu_{\Omega^o}(x)\cdot  \nabla \Big(\tilde{u}_M(x)+S_2(x)\int_{\partial \Omega^o}g^o\,d\sigma\Big)\, d\sigma_x\\
&-\int_{\partial \omega^i} \tilde{u}_m(t) \nu_{\omega^i}(t)\cdot  \nabla \tilde{u}_m(t)\, d\sigma_t
\end{split}
\end{equation}
and
\begin{equation}\label{eq:rep2:5}
E_2[0,0,l_0,\eta_0,r_0]=-\frac{1}{2\pi}\Big(\int_{\partial \Omega^o}g^o\,d\sigma\Big)^2\, . 
\end{equation}
\end{theorem}
\begin{proof}
We set
\[
c_\epsilon\equiv \frac{\log \epsilon}{2\pi}Z_m[\epsilon,\epsilon\delta(\epsilon),\epsilon\delta(\epsilon)\log \epsilon,\eta(\epsilon),\frac{\epsilon}{\rho(\epsilon)}] +\frac{\Xi[\epsilon,\epsilon\delta(\epsilon),\epsilon\delta(\epsilon)\log \epsilon,\eta(\epsilon),\frac{\epsilon}{\rho(\epsilon)}]}{\delta(\epsilon)\epsilon}\qquad \forall \epsilon \in ]0,\epsilon_2[\, .
\]
The Divergence Theorem implies that
\[
\begin{split}
&\int_{\Omega(\epsilon)}|\nabla u(\epsilon,x)|^2\, dx=\int_{\Omega(\epsilon)}|\nabla\big( u(\epsilon,x)-c_\epsilon\big)|^2\, dx \\
&=\int_{\partial \Omega^o}\big(u(\epsilon,x)-c_\epsilon\big)\frac{\partial}{\partial \nu_{\Omega^o}}\big(u(\epsilon,x)-c_\epsilon\big)\, d\sigma_x-\int_{\partial \epsilon \omega^i}\big(u(\epsilon,x)-c_\epsilon\big)\frac{\partial}{\partial \nu_{\epsilon \omega^i}}\big(u(\epsilon,x)-c_\epsilon\big)\, d\sigma_x\\
&=\int_{\partial \Omega^o}\big(u(\epsilon,x)-c_\epsilon\big)\frac{\partial}{\partial \nu_{\Omega^o}}\big(u(\epsilon,x)-c_\epsilon\big)\, d\sigma_x- \int_{\partial  \omega^i}\big(u(\epsilon,\epsilon t)-c_\epsilon\big)\nu_{\omega^i}(t)\cdot  \nabla_t \big(u(\epsilon,\epsilon t)-c_\epsilon\big)\, d\sigma_t\, ,\\
\end{split}
\]
for all $\epsilon \in ]0,\epsilon_2[$. Then we take $U_M$ and $\epsilon_M$ as in Theorem \ref{thm:rep2}, with
\[
\Omega_M \equiv \Omega^o \setminus \overline{\mathbb{B}_2(0,r_M)}\ ,
\]
for some $r_M>0$ such that  $\overline{\mathbb{B}_2(0,r_M)} \subseteq \Omega^o$. We verify that if $\epsilon \in ]0,\epsilon_M[$
\[
\begin{split}
&\int_{\partial \Omega^o}\big(u(\epsilon,x)-c_\epsilon\big)\frac{\partial}{\partial \nu_{\Omega^o}}\big(u(\epsilon,x)-c_\epsilon\big)\, d\sigma_x\\
&=\int_{\partial \Omega^o} U_M[\epsilon,\epsilon\delta(\epsilon),\epsilon\delta(\epsilon)\log \epsilon, \eta(\epsilon),\frac{\epsilon}{\rho(\epsilon)}](x) \nu_{\Omega^o}(x)\cdot  \nabla U_M[\epsilon,\epsilon\delta(\epsilon),\epsilon\delta(\epsilon)\log \epsilon, \eta(\epsilon),\frac{\epsilon}{\rho(\epsilon)}](x)\, d\sigma_x\, \\
&-\frac{\log \epsilon}{2\pi}Z_m[\epsilon,\epsilon\delta(\epsilon),\epsilon\delta(\epsilon)\log \epsilon,\eta(\epsilon),\frac{\epsilon}{\rho(\epsilon)}]\int_{\partial \Omega^o}  \nu_{\Omega^o}(x)\cdot  \nabla U_M[\epsilon,\epsilon\delta(\epsilon),\epsilon\delta(\epsilon)\log \epsilon, \eta(\epsilon),\frac{\epsilon}{\rho(\epsilon)}](x)\, d\sigma_x
\end{split}
\]
Similarly, if $U_m$ and $\epsilon_m$ are as in Theorem \ref{thm:rep2micro}, with
\[
\Omega_m \equiv {\mathbb{B}_2(0,r_m)}\setminus \overline{\omega^i}\ ,
\]
for some $r_m>0$ such that  ${\mathbb{B}_2(0,r_m)} \supseteq \overline{\omega^i}$, then if $\epsilon \in ]0,\epsilon_m[$
\[
\begin{split}
& \int_{\partial  \omega^i}\big(u(\epsilon,\epsilon t)-c_\epsilon\big)\nu_{\omega^i}(t)\cdot  \nabla_t \big(u(\epsilon,\epsilon t)-c_\epsilon\big)\, d\sigma_t\\
&=\int_{\partial \omega^i} U_m[\epsilon,\epsilon\delta(\epsilon),\epsilon\delta(\epsilon)\log \epsilon, \eta(\epsilon),\frac{\epsilon}{\rho(\epsilon)}](t) \nu_{\omega^i}(t)\cdot  \nabla U_m[\epsilon,\epsilon\delta(\epsilon),\epsilon\delta(\epsilon)\log \epsilon, \eta(\epsilon),\frac{\epsilon}{\rho(\epsilon)}](t)\, d\sigma_t\, .
\end{split}
\]
Therefore, we set  $\epsilon_e \equiv \min \{\epsilon_M,\epsilon_m\}$ and
\[
\begin{split}
E_1[\epsilon,\gamma_1,\gamma_2,\gamma_3, \gamma_4]\equiv &\int_{\partial \Omega^o} U_M[\epsilon,\gamma_1,\gamma_2,\gamma_3, \gamma_4](x) \nu_{\Omega^o}(x)\cdot  \nabla U_M[\epsilon,\gamma_1,\gamma_2,\gamma_3, \gamma_4](x)\, d\sigma_x\\
&-\int_{\partial \omega^i} U_m[\epsilon,\gamma_1,\gamma_2,\gamma_3, \gamma_4](t) \nu_{\omega^i}(t)\cdot  \nabla U_m[\epsilon,\gamma_1,\gamma_2,\gamma_3, \gamma_4](t)\, d\sigma_t
\end{split}
\]
and
\[
E_2[\epsilon,\gamma_1,\gamma_2,\gamma_3, \gamma_4] \equiv -\frac{1}{2\pi}Z_m[\epsilon,\gamma_1,\gamma_2,\gamma_3, \gamma_4]\int_{\partial \Omega^o}  \nu_{\Omega^o}(x)\cdot  \nabla U_M[\epsilon,\gamma_1,\gamma_2,\gamma_3, \gamma_4](x)\, d\sigma_x
\]
for all $(\epsilon,\gamma_1,\gamma_2,\gamma_3, \gamma_4)\in ]-\epsilon_e,\epsilon_e[\times \mathcal{U}$. We verify that the maps $E_1$ and $E_2$ are real analytic from $]-\epsilon_e,\epsilon_e[\times \mathcal{U}$ to $\mathbb{R}$ and that equality \eqref{eq:rep2:3} holds.  Moreover, we also have
\[
\begin{split}
E_1[0,0,l_0,\eta_0,r_0]&= \int_{\partial \Omega^o} \Big(\tilde{u}_M(x)+S_2(x)\int_{\partial \Omega^o}g^o\,d\sigma\Big) \nu_{\Omega^o}(x)\cdot  \nabla \Big(\tilde{u}_M(x)+S_2(x)\int_{\partial \Omega^o}g^o\,d\sigma\Big)\, d\sigma_x\\
&-\int_{\partial \omega^i} \tilde{u}_m(t) \nu_{\omega^i}(t)\cdot  \nabla \tilde{u}_m(t)\, d\sigma_t
\end{split}
\]
and
\[
\begin{split}
E_2[0,0,l_0,\eta_0,r_0]&= -\frac{1}{2\pi}\int_{\partial \Omega^o}g^o\,d\sigma \int_{\partial \Omega^o} \nu_{\Omega^o}(x)\cdot  \nabla \Big(\tilde{u}_M(x)+S_2(x)\int_{\partial \Omega^o}g^o\, d\sigma \Big)\, d\sigma_x\\
&= -\frac{1}{2\pi}\Big(\int_{\partial \Omega^o}g^o\,d\sigma\Big)^2\, ,
\end{split}
\]
and accordingly equalities \eqref{eq:rep2:4} and \eqref{eq:rep2:5} hold.
\end{proof}

\section{Remarks on the linear case}\label{case}

In this section, we make further considerations on the asymptotic behavior of the solution in the linear case as the parameter $\epsilon$ tends to $0$. Clearly, we can apply the results of Section \ref{inteqfor} to the linear case. In particular, if we have
\[
F_\epsilon(\tau)=\tau \qquad \forall (\tau,\epsilon)\in \mathbb{R}\times ]0,\epsilon_0[\, ,
\]
problem \eqref{bvpdelta} reduces to the following linear problem
\begin{equation}
\label{bvplineps}
\left\{
\begin{array}{ll}
\Delta u(x)=0 & \forall x \in \Omega(\epsilon)\,,\\
\frac{\partial}{\partial \nu_{\Omega^o}}u(x)=g^o(x) & \forall x \in \partial \Omega^o\, ,\\
\frac{\partial}{\partial \nu_{\epsilon\omega^i}}u(x)=\delta(\epsilon) u(x)+\frac{g^i(x/\epsilon)}{\rho(\epsilon)} & \forall x \in \epsilon \partial \omega^i\, .
\end{array}
\right.
\end{equation}
We also know that for each $\epsilon \in ]0,\epsilon_0[$, problem \eqref{bvplineps} has a unique solution in $C^{1,\alpha}(\overline{ \Omega(\epsilon)})$, which we denote by $u[\epsilon]$.
Clearly,
\[
\epsilon \delta(\epsilon) F_\epsilon \Big(\frac{1}{\epsilon \delta(\epsilon)}\tau\Big)=\tau \qquad  \forall (\tau,\epsilon)\in \mathbb{R}\times ]0,\epsilon_0[\, ,
\]
and thus we can take for example
\[
\eta(\epsilon)=0 \qquad \forall \epsilon \in ]0,\epsilon_0[\, ,
\]
and
\[
\tilde{F}(\tau,\eta)=\tau \qquad \forall (\tau,\eta)\in \mathbb{R}^2\, .
\]
In particular,
\[
\eta_0=0\, ,
\]
and
\[
\partial_\tau\tilde{F}(\tau,\eta)=1 \qquad \forall (\tau,\eta)\in \mathbb{R}^2\, .
\]
All the assumptions in Sections \ref{inteqfor} and \ref{rep} are satisfied. In particular, the solutions of the corresponding limiting systems exist and are unique (see assumption  \eqref{eq:exsol:2}). In the linear case, equations \eqref{eq:limsys:1a2}-\eqref{eq:limsys:1b2} become
\begin{align}
&-\frac{1}{2}\mu^o(x)+\int_{\partial \Omega^o}\nu_{\Omega^o}(x)\cdot \nabla S_2(x-y)\mu^o(y)\, d\sigma_y\nonumber\\
&\qquad+\nu_{\Omega^o}(x)\cdot \nabla S_2(x)\int_{\partial \omega^i}\mu^i(s)\, d\sigma_s=g^o(x) \qquad \forall x \in \partial \Omega^o\, ,\label{eq:limsyslin:1a2}\\
&\frac{1}{2}\mu^i(t)+\int_{\partial \omega^i}\nu_{\omega^i}(t)\cdot \nabla S_2(t-s)\mu^i(s)\, d\sigma_s \nonumber\\
&\qquad=\frac{l_0}{2\pi} \int_{\partial \omega^i}\mu^i\, d\sigma+\xi+g^i(t)r_0 \qquad \forall t \in \partial \omega^i\, .\label{eq:limsyslin:1b2}
\end{align}
By arguing as in the proof of Proposition \ref{prop:Lmbd2}, one can prove that the system  \eqref{eq:limsyslin:1a2}-\eqref{eq:limsyslin:1b2} in the unknown $(\mu^o,\mu^i,\xi)$ admits a unique solution $(\tilde{\mu}^o,\tilde{\mu}^i,\tilde{\xi})$ in $C^{0,\alpha}(\partial \Omega^o)_0\times C^{0,\alpha}(\partial \omega^i) \times \mathbb{R}$. In particular, by integrating \eqref{eq:limsyslin:1a2}, we recall that we obtain
\[
\int_{\partial \omega^i}\tilde{\mu}^i(s)\, d\sigma_s=\int_{\partial \Omega^o}g^o(x)\, d\sigma_x\, .
\]
By integrating \eqref{eq:limsyslin:1b2}, we deduce that
\[
\int_{\partial \Omega^o}g^o(x)\, d\sigma_x=|\partial \omega^i|_{1} \frac{l_0}{2\pi}\int_{\partial \Omega^o}g^o(x)\, d\sigma_x +|\partial \omega^i|_{1}\tilde{\xi}+r_0\int_{\partial \omega^i}g^i(t)\,d\sigma_t\, ,
\]
which implies
\[
\tilde{\xi}=\frac{1}{|\partial \omega^i|_{1}}\Bigg(\Big(1-|\partial \omega^i|_{1} \frac{l_0}{2\pi}\Big)\int_{\partial \Omega^o}g^o(x)\, d\sigma_x-r_0\int_{\partial \omega^i}g^i(t)\,d\sigma_t\Bigg)\, .
\]
We note that in case $\Omega^o=\omega^i=\mathbb{B}_2(0,1)$ and
\[
g^o(x)=a \qquad \forall x \in \partial \mathbb{B}_2(0,1)\, ,\qquad g^i(t)=b \qquad \forall t \in \partial \mathbb{B}_2(0,1)\, ,
\]
we obtain
\[
\begin{split}
\tilde{\xi}&=\Bigg(\Big(1-2\pi \frac{l_0}{2\pi}\Big)a-b r_0 \Bigg)\\
&=\Bigg(a- a l_0-b r_0 \Bigg)\, .
\end{split}
\]
Therefore, we recover also the results of Section \ref{model}.

\section*{Acknowledgement}
The authors acknowledge the support from EU through the H2020-MSCA-RISE-2020 project EffectFact,
Grant agreement ID: 101008140. The authors thank Dr.~Luigi Provenzano for valuable discussions on Steklov eigenvalues in relation to the toy problem of Section \ref{model}. P.M.~also acknowledges the support of  the SPIN Project  ``DOMain perturbation problems and INteractions Of scales - DOMINO''  of the Ca' Foscari University of Venice. Part of the work was done while P.M.~was visiting M.D.~at Rockfield Software Limited. P.M.~wishes to thank M.D.~and Rockfield Software Limited for the kind hospitality. P.M.~is a  member of the Gruppo Nazionale per l'Analisi Matematica, la Probabilit\`a e le loro Applicazioni (GNAMPA) of the Istituto Nazionale di Alta Matematica (INdAM). G.M.~acknowledges also Ser Cymru Future Generation Industrial Fellowship number AU224 -- 80761. M.D. also acknowledges  the Royal Academy of Engineering for the Industrial Fellowship.


\end{document}